\documentclass{amsart}
\usepackage{setspace}
\usepackage{a4}
\usepackage{amsthm}
\usepackage{latexsym}
\usepackage{amsfonts}
\usepackage{graphicx}
\usepackage{textcomp}
\usepackage{cite}
\usepackage{enumerate}
\usepackage{amssymb}
\usepackage{hyperref}
\usepackage{amsmath}
\usepackage{tikz}
\usepackage[mathscr]{euscript}
\usepackage{mathtools}
\newtheorem{theorem}{Theorem}[section]

\newtheorem{corollary}[theorem] {Corollary}
\newtheorem{definition}[theorem]{Definition}
\newtheorem{example}[theorem]{Example}

\newtheorem{problem}[theorem]{Problem}

\title{This is the title}
\usepackage{fancyhdr}

\begin{document}
\hrule\hrule\hrule\hrule\hrule
\vspace{0.3cm}	
\begin{center}
{\bf{UNBOUNDED DONOHO-STARK-ELAD-BRUCKSTEIN-RICAUD-TORR\'{E}SANI  UNCERTAINTY PRINCIPLES}}\\
\vspace{0.3cm}
\hrule\hrule\hrule\hrule\hrule
\vspace{0.3cm}
\textbf{K. MAHESH KRISHNA}\\
Post Doctoral Fellow \\
Statistics and Mathematics Unit\\
Indian Statistical Institute, Bangalore Centre\\
Karnataka 560 059, India\\
Email:  kmaheshak@gmail.com\\

Date: \today
\end{center}

\hrule\hrule
\vspace{0.5cm}
\textbf{Abstract}: Let $(\Omega, \mu)$,  $(\Delta, \nu)$ be   measure spaces and $p=1$ or $p=\infty$. Let  $(\{f_\alpha\}_{\alpha\in \Omega}, \{\tau_\alpha\}_{\alpha\in \Omega})$  and   $(\{g_\beta\}_{\beta\in \Delta}, \{\omega_\beta\}_{\beta\in \Delta})$   be	unbounded continuous p-Schauder frames  for a Banach space $\mathcal{X}$. Then for every $x \in ( \mathcal{D}(\theta_f) \cap\mathcal{D}(\theta_g))\setminus\{0\}$,  we show that 
\begin{align}\label{UB}
	\mu(\operatorname{supp}(\theta_f x))\nu(\operatorname{supp}(\theta_g x))	 \geq 	\frac{1}{\left(\displaystyle\sup_{\alpha \in \Omega, \beta \in \Delta}|f_\alpha(\omega_\beta)|\right)\left(\displaystyle\sup_{\alpha \in \Omega , \beta \in \Delta}|g_\beta(\tau_\alpha)|\right)},
\end{align}
where 
\begin{align*}
	&\theta_f:\mathcal{D}(\theta_f)   \ni x \mapsto \theta_fx \in \mathcal{L}^p(\Omega, \mu); \quad   \theta_fx: \Omega \ni \alpha \mapsto  (\theta_fx) (\alpha):= f_\alpha (x) \in \mathbb{K},\\
	&\theta_g: \mathcal{D}(\theta_g) \ni x \mapsto \theta_gx \in \mathcal{L}^p(\Delta, \nu); \quad   \theta_gx: \Delta \ni \beta \mapsto  (\theta_gx) (\beta):= g_\beta (x) \in \mathbb{K}.
\end{align*}
We call Inequality (\ref{UB}) as \textbf{Unbounded Donoho-Stark-Elad-Bruckstein-Ricaud-Torr\'{e}sani Uncertainty Principle}. Along with recent  \textbf{Functional Continuous Uncertainty Principle}  [arXiv:2308.00312], Inequality (\ref{UB}) also improves Ricaud-Torr\'{e}sani uncertainty principle \textit{[IEEE Trans. Inform. Theory, 2013]}. In particular, it improves  Elad-Bruckstein uncertainty principle \textit{[IEEE Trans. Inform. Theory, 2002]} and Donoho-Stark uncertainty principle \textit{[SIAM J. Appl. Math., 1989]}. 



\hrule

\hrule
\section{Introduction}
Given a collection $\{\tau_j\}_{j=1}^n$ in a finite dimensional Hilbert space $\mathcal{H}$ over $\mathbb{K}$ ($\mathbb{R}$ or $\mathbb{C}$), define
\begin{align*}
	\theta_\tau: \mathcal{H} \ni h \mapsto \theta_\tau h \coloneqq (\langle h, \tau_j\rangle)_{j=1}^n \in \mathbb{K} ^n.
\end{align*}
Most general form of discrete uncertainty principle for finite dimensional Hilbert spaces is the following.
\begin{theorem} (\textbf{Donoho-Stark-Elad-Bruckstein-Ricaud-Torr\'{e}sani Uncertainty Principle}) \cite{DONOHOSTARK, ELADBRUCKSTEIN, RICAUDTORRESANI} \label{RT}
	Let $\{\tau_j\}_{j=1}^n$,  $\{\omega_j\}_{j=1}^n$ be two Parseval frames   for a  finite dimensional Hilbert space $\mathcal{H}$. Then 
\begin{align*}
	\left(\frac{\|\theta_\tau h\|_0+\|\theta_\omega h\|_0}{2}\right)^2	\geq \|\theta_\tau h\|_0\|\theta_\omega h\|_0\geq \frac{1}{\displaystyle\max_{1\leq j, k \leq n}|\langle\tau_j, \omega_k \rangle|^2}, \quad \forall h \in \mathcal{H}\setminus \{0\}.
		\end{align*}	
\end{theorem}
Recently, Theorem \ref{RT} has been derived for Banach spaces using the following notions.
\begin{definition}\cite{KRISHNA2, KRISHNA3}\label{PCSF}
	Let 	$(\Omega, \mu)$ be a measure space. Let    $\{\tau_\alpha\}_{\alpha\in \Omega}$ be a collection in a Banach   space $\mathcal{X}$ and     $\{f_\alpha\}_{\alpha\in \Omega}$ be a collection in  $\mathcal{X}^*$. The pair $(\{f_\alpha\}_{\alpha\in \Omega}, \{\tau_\alpha\}_{\alpha\in \Omega})$   is said to be a \textbf{continuous p-Schauder frame}  for $\mathcal{X}$   ($1\leq p\leq \infty$) if the following holds. 	
	\begin{enumerate}[\upshape(i)]
		\item For every $x\in \mathcal{X}$, the map $\Omega \ni \alpha \mapsto f_\alpha(x)\in \mathbb{K}$ 	is measurable.
			\item The map 
		\begin{align*}
			\theta_f: \mathcal{X} \ni x \mapsto \theta_fx \in \mathcal{L}^p(\Omega, \mu); \quad   \theta_fx: \Omega \ni \alpha \mapsto  (\theta_fx) (\alpha)\coloneqq f_\alpha (x) \in \mathbb{K}
		\end{align*}
		is a well-defined linear isometry. 
		\item For every $x\in \mathcal{X}$, the map 
		$
		\Omega \ni \alpha \mapsto f_\alpha(x)\tau_\alpha \in \mathcal{X}$
		is weakly measurable.
		\item For every $x \in \mathcal{X}$, 
		\begin{align*}
			x=\int\limits_{\Omega}	f_\alpha (x)\tau_\alpha \, d \mu(\alpha),
		\end{align*}  
		where the 	integral is weak integral.
	\end{enumerate}
\end{definition}
\begin{theorem}(\textbf{Functional Continuous  Donoho-Stark-Elad-Bruckstein-Ricaud-Torr\'{e}sani Uncertainty Principle}) \cite{KRISHNA2, KRISHNA1, KRISHNA3} \label{MT}
	Let $(\Omega, \mu)$,  $(\Delta, \nu)$ be   measure spaces. Let  $(\{f_\alpha\}_{\alpha\in \Omega}$, $ \{\tau_\alpha\}_{\alpha\in \Omega})$  and   $(\{g_\beta\}_{\beta\in \Delta}$, $ \{\omega_\beta\}_{\beta\in \Delta})$   be	continuous p-Schauder frames  for a Banach space $\mathcal{X}$. Then  
	\begin{enumerate}[\upshape(i)]
		\item for $1<p<\infty $, we have 
		\begin{align*}
			&\mu(\operatorname{supp}(\theta_f x))^\frac{1}{p}	\nu(\operatorname{supp}(\theta_g x))^\frac{1}{q} \geq 	\frac{1}{\displaystyle\sup_{\alpha \in \Omega, \beta \in \Delta}|f_\alpha(\omega_\beta)|}, \quad \forall x \in \mathcal{X}\setminus\{0\};\\ 	&\nu(\operatorname{supp}(\theta_g x))^\frac{1}{p}	\mu(\operatorname{supp}(\theta_f x))^\frac{1}{q}\geq \frac{1}{\displaystyle\sup_{\alpha \in \Omega , \beta \in \Delta}|g_\beta(\tau_\alpha)|},\quad \forall x \in \mathcal{X}\setminus\{0\}.
		\end{align*}
	where $q$ is the  conjugate index of $p$.
		\item for $p=1$, we have 
		\begin{align}\label{ONE}
			\mu(\operatorname{supp}(\theta_f x)) \geq 	\frac{1}{\displaystyle\sup_{\alpha \in \Omega, \beta \in \Delta}|f_\alpha(\omega_\beta)|}, \quad  	\nu(\operatorname{supp}(\theta_g x))\geq \frac{1}{\displaystyle\sup_{\alpha \in \Omega , \beta \in \Delta}|g_\beta(\tau_\alpha)|}.
		\end{align}
		\item for $p=\infty$, we have 
		\begin{align*}
			\mu(\operatorname{supp}(\theta_f x)) \geq 	\frac{1}{\displaystyle\sup_{\alpha \in \Omega, \beta \in \Delta}|g_\beta(\tau_\alpha)|}, \quad  	\nu(\operatorname{supp}(\theta_g x))\geq \frac{1}{\displaystyle\sup_{\alpha \in \Omega , \beta \in \Delta}|f_\alpha(\omega_\beta)|}.	
		\end{align*}
	\end{enumerate}
\end{theorem}
In this paper,  we derive an unbounded  uncertainty principle which  contains Theorem \ref{RT} as a particular case.

\section{Unbounded Donoho-Stark-Elad-Bruckstein-Ricaud-Torr\'{e}sani Uncertainty Principle}
We first generalize  Definition  \ref{PCSF}  using unbounded linear functionals. Our motivation to do so is the theory of unbounded frames (also known as semi frames or pseudo frames) for Hilbert and Banach spaces, see \cite{LIULIUZHENG, CHRISTENSEN, ANTOINEBALAZS, ANTOINEBALAZS2}.
 \begin{definition}\label{CPSF}
	Let 	$(\Omega, \mu)$ be a measure space and $1\leq p \leq \infty$. Let    $\{\tau_\alpha\}_{\alpha\in \Omega}$ be a collection in a Banach   space $\mathcal{X}$ and     $\{f_\alpha\}_{\alpha\in \Omega}$ be a collection of linear functions on $\mathcal{X}$ (which may not be bounded). The pair $(\{f_\alpha\}_{\alpha\in \Omega}, \{\tau_\alpha\}_{\alpha\in \Omega})$   is said to be an \textbf{unbounded continuous p-Schauder frame} or \textbf{continuous semi p-Schauder frame} for $\mathcal{X}$    if the following conditions  holds. 	
	\begin{enumerate}[\upshape(i)]
		\item For every $x\in \mathcal{X}$, the map $\Omega \ni \alpha \mapsto f_\alpha(x)\in \mathbb{K}$ 	is measurable.
		\item The map 
		\begin{align*}
			\theta_f: \mathcal{D}(\theta_f) \ni x \mapsto \theta_fx \in \mathcal{L}^p(\Omega, \mu); \quad   \theta_fx: \Omega \ni \alpha \mapsto  (\theta_fx) (\alpha)\coloneqq f_\alpha (x) \in \mathbb{K}
		\end{align*}
		is well-defined (need not be bounded).
		\item For every $x\in \mathcal{X}$, the map 
		$
			\Omega \ni \alpha \mapsto f_\alpha(x)\tau_\alpha \in \mathcal{X}$
		is weakly measurable.
		 	\item For every $x \in \mathcal{D}(\theta_f)$, 
		\begin{align*}
			x=\int\limits_{\Omega}	f_\alpha (x)\tau_\alpha \, d \mu(\alpha),
		\end{align*}  
	where the 	integral is weak integral.
	\end{enumerate}
\end{definition}
\begin{example}
Let $1\leq p <\infty$. Define $ \mathcal{X}\coloneqq \ell^p(\mathbb{N})$ and let $\{e_n\}_n$ be the standard Schauder basis for $	 \mathcal{X}$. Let $\{\zeta_n\}_n$ be the coordinate functionals associated with $\{e_n\}_n$. Define 
\begin{align*}
	\tau_n\coloneqq \frac{e_n}{n}, \quad f_n\coloneqq n\zeta_n, \quad \forall n \in \mathbb{N}.
\end{align*}
Then 
\begin{align*}
\mathcal{D}(\theta_f)=\left\{\{a_n\}_n\in \ell^p(\mathbb{N}): \sum_{n=1}^{\infty}|na_n|^p<\infty\right\}
\end{align*}
and 
\begin{align*}
	\sum_{n=1}^{\infty}f_n(\{a_m\}_m)\tau_n=\sum_{n=1}^{\infty}na_n\frac{e_n}{n}=\{a_n\}_n, \quad \forall \{a_n\}_n \in \mathcal{D}(\theta_f).
\end{align*}
Therefore $(\{f_n\}_n, \{\tau_n\}_n)$  is an unbounded p-Schauder frame for $\mathcal{X}$. 
\end{example}
\begin{theorem}(\textbf{1-Unbounded Donoho-Stark-Elad-Bruckstein-Ricaud-Torr\'{e}sani Uncertainty Principle})\label{UUP}
Let $(\Omega, \mu)$,  $(\Delta, \nu)$ be   measure spaces. Let  $(\{f_\alpha\}_{\alpha\in \Omega}, \{\tau_\alpha\}_{\alpha\in \Omega})$  and   $(\{g_\beta\}_{\beta\in \Delta}, \{\omega_\beta\}_{\beta\in \Delta})$   be	unbounded continuous 1-Schauder frames  for a Banach space $\mathcal{X}$. Then for every $x \in (\mathcal{D}(\theta_f)\cap \mathcal{D}(\theta_g))\setminus\{0\}$,  we have 
\begin{align}\label{UPT}
		\mu(\operatorname{supp}(\theta_f x))\nu(\operatorname{supp}(\theta_g x)) \geq 	\frac{1}{\left(\displaystyle\sup_{\alpha \in \Omega, \beta \in \Delta}|f_\alpha(\omega_\beta)|\right)\left(\displaystyle\sup_{\alpha \in \Omega , \beta \in \Delta}|g_\beta(\tau_\alpha)|\right)}.
\end{align}
\end{theorem}
\begin{proof}
	Let $x \in \mathcal{D}(\theta_f)\setminus\{0\}$. Then
		\begin{align*}
	\|\theta_fx\|&=\int\limits_{\Omega}|f_\alpha (x)|\, d\mu(\alpha)=\int\limits_{ \operatorname{supp}(\theta_fx)}|f_\alpha(x)|\, d\mu(\alpha)=\int\limits_{\operatorname{supp}(\theta_fx)}\left|f_\alpha\left(\int\limits_{\Delta}g_\beta(x)\omega_\beta\, d\nu(\beta)\right)\right|\, d\mu(\alpha)\\
	&=\int\limits_{\operatorname{supp}(\theta_fx)}\left|\int\limits_{\Delta}g_\beta(x)f_\alpha(\omega_\beta)\, d\nu(\beta)\right|\, d\mu(\alpha)=\int\limits_{\operatorname{supp}(\theta_fx)}\left|\int\limits_{ \operatorname{supp}(\theta_gx)}g_\beta(x)f_\alpha(\omega_\beta)\, d\nu(\beta)\right|\, d\mu(\alpha)\\
	&\leq \int\limits_{ \operatorname{supp}(\theta_fx)}\int\limits_{ \operatorname{supp}(\theta_gx)}|g_\beta(x)f_\alpha(\omega_\beta)|\, d\nu(\beta)\, d\mu(\alpha)\\
	&\leq \left(\displaystyle\sup_{\alpha \in \Omega , \beta \in \Delta}|f_\alpha(\omega_\beta)|\right)\int\limits_{ \operatorname{supp}(\theta_fx)}\int\limits_{\operatorname{supp}(\theta_gx)}|g_\beta(x)|\, d\nu(\beta)\, d\mu(\alpha)\\
	&=\left(\displaystyle\sup_{\alpha \in \Omega , \beta \in \Delta}|f_\alpha(\omega_\beta)|\right) 	\mu(\operatorname{supp}(\theta_f x))\int\limits_{\operatorname{supp}(\theta_gx)}|g_\beta(x)|\, d\nu(\beta)\\
	&=\left(\displaystyle\sup_{\alpha \in \Omega , \beta \in \Delta}|f_\alpha(\omega_\beta)|\right)	\mu(\operatorname{supp}(\theta_f x))\|\theta_gx\|.
	\end{align*}
Therefore 
\begin{align}\label{FI}
	\frac{1}{\displaystyle\sup_{\alpha \in \Omega , \beta \in \Delta}|f_\alpha(\omega_\beta)|}\|\theta_fx\|\leq 	\mu(\operatorname{supp}(\theta_f x))\|\theta_gx\|.
\end{align}
On the other way, let $x \in \mathcal{D}(\theta_g)\setminus\{0\}$. Then
\begin{align*}
	\|\theta_gx\|&=\int\limits_{\Delta}|g_\beta(x)|\, d\nu(\beta)=\int\limits_{ \operatorname{supp}(\theta_gx)}|g_\beta(x)|\, d\nu(\beta)
	=\int\limits_{\operatorname{supp}(\theta_gx)}\left|g_\beta\left(\int\limits_{\Omega}f_\alpha(x)\tau_\alpha\, d\mu(\alpha)\right)\right|\, d\nu(\beta)\\
	&
	=\int\limits_{\operatorname{supp}(\theta_gx)}\left|\int\limits_{\Omega}f_\alpha(x)g_\beta(\tau_\alpha)\, d\mu(\alpha)\right|\, d\nu(\beta)=\int\limits_{\operatorname{supp}(\theta_gx)}\left|\int\limits_{ \operatorname{supp}(\theta_fx)}f_\alpha(x)g_\beta(\tau_\alpha)\, d\mu(\alpha)\right|\, d\nu(\beta)\\
	&\leq \int\limits_{ \operatorname{supp}(\theta_gx)}\int\limits_{ \operatorname{supp}(\theta_fx)}|f_\alpha(x)g_\beta(\tau_\alpha)|\, d\mu(\alpha)\, d\nu(\beta)\\
	&\leq \left(\displaystyle\sup_{\alpha \in \Omega , \beta \in \Delta}|g_\beta(\tau_\alpha)|\right)\int\limits_{ \operatorname{supp}(\theta_gx)}\int\limits_{ \operatorname{supp}(\theta_fx)}|f_\alpha(x)|\, d\mu(\alpha)\, d\nu(\beta)\\
	&=\left(\displaystyle\sup_{\alpha \in \Omega , \beta \in \Delta}|g_\beta(\tau_\alpha)|\right)\nu(\operatorname{supp}(\theta_g x))\int\limits_{ \operatorname{supp}(\theta_fx)}|f_\alpha(x)|\, d\mu(\alpha)\\
	&=\left(\displaystyle\sup_{\alpha \in \Omega , \beta \in \Delta}|g_\beta(\tau_\alpha)|\right)\nu(\operatorname{supp}(\theta_g x))\|\theta_f x\|.
\end{align*}
Therefore 
\begin{align}\label{SI}
	\frac{1}{\displaystyle\sup_{\alpha \in \Omega , \beta \in \Delta}|g_\beta(\tau_\alpha)|}\|\theta_gx\|\leq \nu(\operatorname{supp}(\theta_g x))\|\theta_fx\|.
\end{align}	
Multiplying Inequalities (\ref{FI}) and (\ref{SI}) we get 
\begin{align*}
\frac{1}{\left(\displaystyle\sup_{\alpha \in \Omega , \beta \in \Delta}|f_\alpha(\omega_\beta)|\right)\left(\displaystyle\sup_{\alpha \in \Omega , \beta \in \Delta}|g_\beta(\tau_\alpha)|\right)}\|\theta_fx\|\|\theta_gx\|&\leq 	\mu(\operatorname{supp}(\theta_f x))\nu(\operatorname{supp}(\theta_g x))\|\theta_gx\|\|\theta_fx\|,\\
& \quad \forall x \in (\mathcal{D}(\theta_f)\cap \mathcal{D}(\theta_g))\setminus\{0\}.
\end{align*}
A cancellation of $\|\theta_fx\|\|\theta_gx\|$ gives the required inequality.
\end{proof}
\begin{corollary}
Let $(\{f_j\}_{j=1}^n, \{\tau_j\}_{j=1}^n)$ and $(\{g_k\}_{k=1}^m, \{\omega_k\}_{k=1}^m)$ be  collections in a Banach space $\mathcal{X}$ such that 
\begin{align*}
	x=\sum_{j=1}^nf_j(x)\tau_j=\sum_{k=1}^mg_k(x)\omega_k, \quad \forall x \in \mathcal{X}.	
\end{align*}
Then for every $x \in \mathcal{X}\setminus\{0\}$,  
\begin{align*}
	\|\theta_f x\|_0\|\theta_g x\|_0 \geq 	\frac{1}{\left(\displaystyle\max_{1\leq j\leq n, 1\leq k\leq m}|f_j(\omega_k)|\right)\left(\displaystyle\max_{1\leq j\leq n, 1\leq k\leq m}|g_k(\tau_j)|\right)},
\end{align*}
where 
\begin{align*}
	\theta_f: \mathcal{X} \ni x \mapsto (f_j(x) )_{j=1}^n \in \ell^1([n]); \quad \theta_g: \mathcal{X} \ni x \mapsto (g_k(x) )_{k=1}^m \in \ell^1([m]).	
\end{align*}
\end{corollary}
Even though by multiplying two inequalities in (\ref{ONE}) we get Inequality (\ref{UPT}) for continuous 1-Schauder frames, observe that the conclusion in (\ref{ONE}) is stronger (with stronger assumption) than that of Theorem \ref{UUP}.
\begin{theorem}
(\textbf{$\infty$-Unbounded Donoho-Stark-Elad-Bruckstein-Ricaud-Torr\'{e}sani Uncertainty Principle})
Let $(\Omega, \mu)$,  $(\Delta, \nu)$ be   measure spaces. Let  $(\{f_\alpha\}_{\alpha\in \Omega}, \{\tau_\alpha\}_{\alpha\in \Omega})$  and   $(\{g_\beta\}_{\beta\in \Delta}, \{\omega_\beta\}_{\beta\in \Delta})$   be	unbounded continuous $\infty$-Schauder frames  for a Banach space $\mathcal{X}$. Then for every $x \in (\mathcal{D}(\theta_f)\cap \mathcal{D}(\theta_g))\setminus\{0\}$,  we have 
\begin{align}
	\mu(\operatorname{supp}(\theta_f x))\nu(\operatorname{supp}(\theta_g x)) \geq 	\frac{1}{\left(\displaystyle\sup_{\alpha \in \Omega, \beta \in \Delta}|f_\alpha(\omega_\beta)|\right)\left(\displaystyle\sup_{\alpha \in \Omega , \beta \in \Delta}|g_\beta(\tau_\alpha)|\right)}.
\end{align}	
\end{theorem}
\begin{proof}
Let $x \in  \mathcal{D}(\theta_f)\setminus\{0\}$ and $\alpha \in \Omega$. Then 
\begin{align*}
	|(\theta_fx)(\alpha)|&=|f_\alpha (x)|=\left|f_\alpha \left(\int\limits_{\Delta}g_\beta(x)\omega_\beta\, d\nu(\beta)\right)\right|=\left|\int\limits_{\Delta}g_\beta(x)f_\alpha(\omega_\beta)\, d\nu(\beta)\right|\\
	&=\left|\int\limits_{\operatorname{supp}(\theta_gx)}g_\beta(x)f_\alpha(\omega_\beta)\, d\nu(\beta)\right|\leq \int\limits_{\operatorname{supp}(\theta_gx)}|g_\beta(x)f_\alpha(\omega_\beta)|\, d\nu(\beta)\\
	&\leq \left(\displaystyle\sup_{\alpha \in \Omega , \beta \in \Delta}|f_\alpha(\omega_\beta)|\right)\int\limits_{\operatorname{supp}(\theta_gx)}|g_\beta(x)|\, d\nu(\beta)\\
	&\leq  \left(\displaystyle\sup_{\alpha \in \Omega , \beta \in \Delta}|f_\alpha(\omega_\beta)|\right)\nu(\operatorname{supp}(\theta_g x))\|\theta_g x\|.
\end{align*}
Therefore 
\begin{align}\label{I1}
	\frac{1}{\displaystyle\sup_{\alpha \in \Omega , \beta \in \Delta}|f_\alpha(\omega_\beta)|}\|\theta_fx\|\leq 	\nu(\operatorname{supp}(\theta_g x))\|\theta_gx\|.
\end{align}	
Now let $\beta \in \Delta$. Then 
\begin{align*}
	|(\theta_g x)(\beta)|&=|g_\beta(x)|=\left|g_\beta \left(\int\limits_{\Omega}f_\alpha(x)\tau_\alpha\, d\mu(\alpha)\right)\right|=\left|\int\limits_{\Omega}f_\alpha(x)g_\beta(\tau_\alpha)\, d\mu(\alpha)\right|\\
	&=\left|\int\limits_{\operatorname{supp}(\theta_fx)}f_\alpha(x)g_\beta(\tau_\alpha)\, d\mu(\alpha)\right|\leq \int\limits_{\operatorname{supp}(\theta_fx)}|f_\alpha(x)g_\beta(\tau_\alpha)|\, d\mu(\alpha)\\
	&\leq \left(\displaystyle\sup_{\alpha \in \Omega , \beta \in \Delta}|g_\beta(\tau_\alpha)|\right)\int\limits_{\operatorname{supp}(\theta_fx)}|f_\alpha(x)|\, d\mu(\alpha)\\
	&\leq \left(\displaystyle\sup_{\alpha \in \Omega , \beta \in \Delta}|g_\beta(\tau_\alpha)|\right)\mu(\operatorname{supp}(\theta_fx))\|\theta_f x\|.
\end{align*}
Therefore 
\begin{align}\label{I2}
	\frac{1}{\displaystyle\sup_{\alpha \in \Omega , \beta \in \Delta}|g_\beta(\tau_\alpha)|}\|\theta_gx\|\leq 	\mu(\operatorname{supp}(\theta_f x))\|\theta_fx\|.	
\end{align}
By multiplying Inequalities (\ref{I1}) and (\ref{I2}) and canceling $\|\theta_fx\|\|\theta_gx\|$ we get the required inequality.
\end{proof}

Note that  the proof of Theorem \ref{UUP} does not work for  unbounded continuous p-Schauder frames  for $p>1$ (even by using Holder's inequality). We are therefore  left over with following problem.
\begin{problem}
What is the unbounded version of Theorem  \ref{UUP} for $1<p<\infty$?
\end{problem}

 \bibliographystyle{plain}
 \bibliography{reference.bib}

\begin{thebibliography}{10}

\bibitem{ANTOINEBALAZS}
J.-P. Antoine and P.~Balazs.
\newblock Frames, semi-frames, and {H}ilbert scales.
\newblock {\em Numer. Funct. Anal. Optim.}, 33(7-9):736--769, 2012.

\bibitem{ANTOINEBALAZS2}
Jean-Pierre Antoine and Peter Balazs.
\newblock Frames and semi-frames.
\newblock {\em J. Phys. A}, 44(20):205201, 25, 2011.

\bibitem{CHRISTENSEN}
Ole Christensen.
\newblock Frames and pseudo-inverses.
\newblock {\em J. Math. Anal. Appl.}, 195(2):401--414, 1995.

\bibitem{DONOHOSTARK}
David~L. Donoho and Philip~B. Stark.
\newblock Uncertainty principles and signal recovery.
\newblock {\em SIAM J. Appl. Math.}, 49(3):906--931, 1989.

\bibitem{ELADBRUCKSTEIN}
Michael Elad and Alfred~M. Bruckstein.
\newblock A generalized uncertainty principle and sparse representation in
  pairs of bases.
\newblock {\em IEEE Trans. Inform. Theory}, 48(9):2558--2567, 2002.

\bibitem{KRISHNA3}
K.~Mahesh Krishna.
\newblock Endpoint functional continuous uncertainty principles.
\newblock {\em https://doi.org/10.5281/zenodo.10121545, 14 November}, 2023.

\bibitem{KRISHNA2}
K.~Mahesh Krishna.
\newblock Functional continuous uncertainty principle.
\newblock {\em arXiv:2308.00312v1 [math.FA] 1 August}, 2023.

\bibitem{KRISHNA1}
K.~Mahesh Krishna.
\newblock Functional {D}onoho-{S}tark-{E}lad-{B}ruckstein-{R}icaud-{T}orrésani
  uncertainty principle.
\newblock {\em arXiv: 2304. 03324v1, [math.FA] 5 April}, 2023.

\bibitem{LIULIUZHENG}
Bei Liu, Rui Liu, and Bentuo Zheng.
\newblock Parseval {$p$}-frames and the {F}eichtinger conjecture.
\newblock {\em J. Math. Anal. Appl.}, 424(1):248--259, 2015.

\bibitem{RICAUDTORRESANI}
Benjamin Ricaud and Bruno Torr\'{e}sani.
\newblock Refined support and entropic uncertainty inequalities.
\newblock {\em IEEE Trans. Inform. Theory}, 59(7):4272--4279, 2013.

\end{thebibliography}

\end{document}